\documentclass{article}
\usepackage[utf8]{inputenc}

\usepackage[normalem]{ulem}

\usepackage{amsthm}
\usepackage{amsmath}
\usepackage{amssymb}
\usepackage{color}
\usepackage{bigints}
\usepackage{mathtools} 
\usepackage{extarrows} 
\usepackage{bbm}
\usepackage{hyperref}
\usepackage{centernot}
\usepackage{stmaryrd}
\usepackage{comment}
\usepackage{authblk} 

\newtheorem{theorem}{Theorem}

\newtheorem{proposition}{Proposition}
\newtheorem{remark}{Remark}

\newtheorem{lemma}{Lemma}

\newcommand{\cC}{\mathcal{C}}
\newcommand{\bP}{\mathbb{P}}
\newcommand{\bE}{\mathbb{E}}

\title{Self-switching random walks on  Erdös-Rényi random graphs feel the phase transition}

\author[1]{G. Iacobelli \thanks{giulio@im.ufrj.br}}
\author[1]{G. Ost \thanks{guilhermeost@im.ufrj.br}}
\author[2]{D. Y. Takahashi \thanks{takahashiyd@gmail.com}}

\affil[1]{Universidade Federal do Rio de Janeiro, Rio de Janeiro, Brazil}
\affil[2]{Universidade Federal do Rio Grande do Norte, Natal, Brazil}

\newcommand{\bR}{\mathbb{R}}
\newcommand{\cF}{\mathcal{F}}

\begin{document}

\maketitle

\begin{abstract}
We study random walks on Erdös-Rényi random graphs in which, every time the random walk returns to the starting point,  first an edge probability is independently sampled according to a priori measure $\mu$, and then an Erdös-Rényi random graph is sampled according to that edge probability.   
When the edge probability $p$ does not depend on the size of the graph $n$ (dense case), we show that the proportion of time the random walk spends on different values of $p$ -- {\it occupation measure} --  converges to the a priori measure $\mu$ as $n$ goes to infinity.  
More interestingly,  when $p=\lambda/n$ (sparse case),
we show that the occupation measure converges to a limiting measure with a density that is a function of the survival probability of a Poisson branching process. This limiting measure is supported on the supercritial values for the Erdös-Rényi random graphs, showing that self-witching random walks can detect the phase transition.     
\end{abstract}




\section{Introduction}

In this paper, we consider self-switching random walks  on Erdös-Rényi random graphs (SSRW-ER), which are discrete-time random walks on dynamically evolving random graphs where the graph's dynamics depends on the random walk's return time. The dynamics of an SSRW-ER works as follows: every time the random walk returns to the starting point, first an edge probability is independently sampled according to a probability measure $\mu$ ({\it prior} measure), and then an Erdös-Rényi random graph is sampled according to the sampled edge probability.
Our aim is to understand how the proportion of time  the random walk spends on different values of the edge probability $p$ scales with the size $n$ of the underlying Erdös-Renyi graph.  We do that by studying the limit, as time goes to infinity, of the {\em empirical occupation measure} of the edge parameter for a fixed size $n$ of the graph. This limiting measure, referred to as  {\it occupation measure}, is a probability measure on the set of values of the edge probability  that depends on the prior measure $\mu$ and on the expected return time of a simple random walk on Erdös-Renyi graph of size $n$. We then study the behavior of the occupation measure as the size of the graph $n$ goes to infinity.

We first show that, in the dense case, {\it i.e.},  when the edge probability $p$ does not depend on the size of the graph $n$, the  occupation measure converges to the prior measure $\mu$ as $n$ goes to infinity.  
We then ask how the  occupation measure behaves if the edge-probability scales by  $n^{-1}$ (sparse case), {\it i.e.}, $p=\lambda/n$. We show that in this case, the occupation measure converges to a limiting distribution which depends on
 the survival probability of a Poisson branching process with parameter $np$. In particular, the support of the occupation measure is determined by the parameter at which the phase transition of the underlying Erdös-Renyi random graph is observed, {\it i.e.}, $np > 1$.
 The key ingredient of the proofs is the exact computation of the asymptotic expected return time of a simple random walk to a given vertex in an Erdös-Renyi random graph with $n$ vertices and edge probability $p$, either fixed or scaling with $n$ as $\lambda/n$. 
 In the dense case, the expected return time of a simple random walk on Erdös-Renyi random graphs is asymptotically $n$ for all $p \in (0,1]$. In the sparse case, instead, our computation shows that simple random walks on Erdös-Renyi random graphs in the subcritical ($\lambda<1$) and critical ($\lambda=1$) regimes take on average $o(n)$ steps to return to the starting point, whereas the expected return time is $f(\lambda) \, n$ 
in the supercritial regime ($\lambda>1$),  where the positive  constant $f(\lambda)$ is explicitly computed and depends on the survival probability
of a Poisson branching process of parameter  $\lambda$. We believe these results are of independent interest.  

The study of quantitative aspects of random walks on Erdös-Renyi random graphs has recently received considerable attention.  In \cite{LOWE201481}, 
for example, the authors study the mean hitting time of a random walk starting from the  stationary distribution  ({\em a.k.a.} random target time). Denoting by $H_j$  the random target time of vertex $j$, they proved that on an Erdös-Renyi random graph of size $n$   in the regime where the edge probability is much bigger than the connectivity threshold, {\em i.e., } $p\gg \frac{\log n}{n}$, it holds that   $H_j=n + o(n)$ asymptotically almost surely. In \cite{LOWE2023}, 
it was also proved a Central Limit Theorem for the random target time, namely that $H_j-n$, suitably rescaled converges in distribution to a random variable with normal distribution. Similar results were proved in  \cite{Ott2024} 
for random walks on Erdös-Renyi random graphs when $p$ is fixed and  does not scale with $n$. Specifically, in the latter paper, it was  shown that the hitting time (not just the mean hitting time) between any two different vertices is equal to $n(1+ o(1))$. 
%
%
Recently, the results obtained in \cite{Ott2024} have been extended in \cite{Granet2024Hitting} 
to the regime $\frac{\log n}{n^{(k-1)/k}}\leq p\leq 1-\Omega(\frac{\log^4 n}{n})$ for $k\geq 2$ fixed. 

Our contribution differs from the above-mentioned works since  we study the expected first return time, particularly  in the regime where $p=\lambda/n$, thus below the connectivity threshold of the Erdös-Renyi random graph. Interestingly, our analysis shows that SSRW-ER feels the phase transition since the corresponding expected first return time in the supercritical regime  differs from that in the subcritical and critical regimes.



SSRW-ER are random walks on  dynamically changing random environment, in which the dynamics of the environment depends on the random walk trajectory. Previously, we studied self-switching random walks on simple deterministic graphs parameterized by a parameter that is distributed according to a priori distribution $\mu$ ~\cite{gallo2022self}. We showed that, in several cases, the respective occupation measure converges to a measure whose support is a proper subset of the support of $\mu$; a phenomenon that we call {\it emerging dominance}. These processes were introduced to model some animal behaviors, where the law of the dynamics changes depending on the action/behavior taken by the animal. The SSRW-ER introduced in this article are a natural generalization of self-switching random walks on deterministic graphs. We were motivated by the studies of homing behavior in animals ~\cite{papi2012animal} - once animals leave their nest, most of them have the capacity to come back home. In several cases, a random walk model fits well with the homing data ~\cite{wilkinson1952random, jamon1987effectiveness}. Here, we wanted to better understand some scaling properties of homing-like behavior in a random environment. We were especially interested in understanding whether a specific search strategy would emerge in these models, \emph{i.e.}, if an emerging dominance exists. We prove that in the dense case, we do not observe emerging dominance, whereas in the sparse case, we can observe emerging dominance. Our result suggests that the prevalence of specific homing strategy could be in part explained by emerging dominance phenomena.

The article is organized as follows. In Section~\ref{sec:not} we introduce the notations and state the main results. In Section~\ref{sec:proof} we prove the results. 

\section{Notations and Main Results
}\label{sec:not}

The {\it Erdös-Rényi} random graph with vertex set $\{1,\ldots, n\}=:[n]$ is a random graph in which the edge between each pair of vertices is present with probability $p$ and absent otherwise, independently of all the other edges. The parameter $p\in [0,1]$ is called the {\it edge probability.} We denote the Erdös-Rényi random graph with vertex set $[n]$ and edge probability $p$ by $G(n,p)$.    
We write $\bP_{n,p}$ for the distribution of the random graph $G(n,p)$. 

 For vertices $u,v\in [n]$, we write $v\longleftrightarrow u$ to indicate that there exists a path of present edges connecting $u$ and $v$. We adopt the convention that $v\longleftrightarrow v$ for all $v\in [n]$. The connected component of a vertex  $v\in [n]$ will be denoted by 
\begin{equation*}
\label{def_connected_component}
\cC(v):=\{u\in [n]: v\longleftrightarrow u\}\,.    
\end{equation*}
The size of the cluster $\cC(v)$  is denoted $|\cC(v)|$. The largest connected component, denoted by $\cC_{\rm max}$, is defined as the connected component of any vertex $u\in \arg\max_{v\in [n]}|\cC(v)|$.  
In particular, we have that
\begin{equation*}
\label{def_largest_connected_component}
|\cC_{\rm max}|:=\max_{v\in [n]}|\cC(v)|\,.     
\end{equation*}

Given a realization of an Erdös-Rényi random graph $G(n,p)=g$, we consider a discrete-time simple symmetric random walk 
$X=(X_t)_{t\geq 0}$ on $g$, starting from vertex $1$ ({\it i.e.}, $X_0\equiv 1$). 
The distribution of the random walk $X$ given $G(n,p)=g$ will be denoted by $\bP_{g}$ and the corresponding  expectation  will be denoted by $\bE_{g}$.
Let us introduce  
\[
\tau:=\inf\{t\geq 1: X_t=1\}\,,
\] 
the first time the random walk $X$ returns to vertex $1$. 
Note that the random walk $X$ can only visit vertices belonging to the connected component $\cC(1)$, since it starts from vertex $1$. 
Throughout the paper, we use the convention that $\tau=1$ whenever $\cC(1)=\{1\}$, that is when the vertex $1$ is isolated.
Let $\bE_{g}[\tau]$ be the expected return time of random walk $X$ to vertex $1$ given that $G(n,p)=g$.
When the realization of the Erdös-Rényi random graph $G(n,p)$ is not fixed, the expected return time of the random walk $X$ to vertex $1$, denoted by $\bE_{n,p}\left[\tau\right]$, is given  by 
\begin{equation}
\label{def_undicional_exp_of_return_time_to_1}
\bE_{n,p}\left[\tau\right]=\bE_{n,p}\left[\bE_{G(n,p)}\left[\tau\right]\right]\,.   
\end{equation}

Our goal is to study the large scale limits of the {\it empirical occupation measure} of the {\it state process} associated to SSRW-ER. 
The state process associated to an SSRW-ER can be defined as follows. 
Let $(\Theta, \mathcal{F},\mu$) be a probability space and $g_n:\Theta\to (0,1]$ be a measurable function. 
For each $i\geq 1$, we first choose a parameter $\theta_i$ according to the probability distribution $\mu$, independently of everything else, and define $p_{n,i}=g_n(\theta_i)$. Then, we sample an Erdös-Rényi random graph $G(n,p_{n,i})$ 
and run a discrete-time simple symmetric random walk on the sampled graph, starting from vertex $1$.  
We stop the random walk at the first time $\tau_i$ at which the walker returns to vertex 1. In this way, we obtain a sequence of pairs $\big((\theta_i,\tau_i)\big)_{i\geq 1}$.
Given this sequence of pairs we then build the {\it state} process $(P_{n,t})_{t\geq 1}$ as follows. For each $t\geq 1$, we define $P_{n,t}=\theta_i$ whenever $\sum_{k=0}^{i-1}\tau_k<t\leq \sum_{k=0}^{i}\tau_k$, $i\geq 1$, with the convention that $\tau_0=0$.
Once the state process $(P_{n,t})_{t\geq 1}$ is built, we define its empirical occupation measure as
\begin{equation}
\label{def:empirical_occupation_measure}
\mu_{n,T}(A):=\frac{1}{T}\sum_{t=1}^T\mathbbm{1}_{(P_{n,t}\in A)}, \ A\in \cF, \ T\geq 1\,.
\end{equation} 
In the proposition below we provide the limiting measure, as time $T$ goes to
infinity, of the empirical occupation measure.
\begin{proposition}
\label{prop:LLN_for_the_empirical_occupation_measure}
For every $A\in\cF$, it holds that, almost surely as $T\to\infty$,  
\begin{equation*}
\mu_{n,T}(A)\longrightarrow \mu_{n,\infty}(A):=\frac{\bigintsss_A \bE_{n,g_n(\theta)}[\tau]\mu(d\theta)}{\bigintsss_{\Theta}\bE_{n,g_n(\theta)}[\tau]\mu(d\theta)}\,, 
\end{equation*}
where, for all $p\in (0,1]$, $\bE_{n,p}[\tau]$ is defined in \eqref{def_undicional_exp_of_return_time_to_1}.
\end{proposition}
We call the probability measure $\mu_{n,\infty}$ defined in Proposition \ref{prop:LLN_for_the_empirical_occupation_measure} the {\it occupation measure} of the parameter $\theta$ or simply the occupation measure of $\theta$. In the next result, we study the fluctuations of the empirical occupation measure around its limit. For each $A\in\cF$, let us define
\begin{equation}
\label{def:asymptotic_var_CLT}
\sigma^2_{n,A}:=\frac{\int_{\Theta}(\mathbbm{1}_{A}(\theta)-\mu_{n,\infty}(A))^2\bE_{n,g_n(\theta)}[\tau^2]\mu(d\theta)}{\bigintsss_{\Theta}\bE_{n,g_n(\theta)}[\tau]\mu(d\theta)}\,.
\end{equation}
Using Lemma \ref{lem:exp-moments} one can show that $\sigma^2_{n,A}<\infty$ for every $A\in\cF$.  
\begin{proposition}
\label{prop:CLT_for_the_empirical_occupation_measure}
For every $A\in\cF$,  it holds that, as $T\to\infty$, 
\begin{equation*}
\sqrt{T}\big(\mu_{n,T}(A)-\mu_{n,\infty}(A)\big)\longrightarrow \mathcal{N}(0,\sigma^2_{n,A}) \ \text{in distribution}\,, 
\end{equation*}
where $\mu_{n,\infty}$ and $\sigma^2_{n,A}$ are defined in Proposition \ref{prop:LLN_for_the_empirical_occupation_measure} and \eqref{def:asymptotic_var_CLT}, respectively.
\end{proposition}
The proof of Proposition \ref{prop:LLN_for_the_empirical_occupation_measure} and  Proposition \ref{prop:CLT_for_the_empirical_occupation_measure} are given in Appendix \ref{Sec:proof_of_prop_LLN}.

In the next subsections, we study the asymptotic behavior of the occupation measure $\mu_{n,\infty}$ as the number of vertices $n\to \infty$. We do that in  the dense case  ({\em i.e.}, in $G(n,p)$) as well as in the sparse case ({\em i.e.}, in $G(n,\lambda/n)$).

\subsection{Dense case} 
In the dense case,  the edge probability $p\in (0,1]$ is fixed and it does not scale with the number $n$ of vertices of the Erdös-Rényi random graph.
Setting $\Theta=(0,1]$, $\cF=\mathcal{B}((0,1])$ where $\mathcal{B}((0,1])$ denotes the Borel sigma-algebra of $(0,1]$ and $g_n$ the identity function for all $n\geq 1$, the occupation measure of $p$ is given by
\begin{equation}
\label{def_occupation_measure_G_N_p}
\mu_{n,\infty}(A)= \frac{\bigintsss_A \bE_{n,p}[\tau]\mu(dp)}{\bigintsss_{(0,1]}\bE_{n,p}[\tau]\mu(dp)}\, \ , A\in\mathcal{B}((0,1])\,.
\end{equation} 



In the dense case, {\em i.e.}, for $p\in (0,1]$  independent from $n$, it is well known that the $G(n,p)$ is connected with high probability when $n$ tends to infinity. On this latter event, we have that $\bE_{G(n,p)}\left[\tau\right]= \sum_{i \in [n]}d(i)/d(1)$, where $d(i)$ denotes the degree of vertex $i$ in $G(n,p)$. Since the degree of any vertex is distributed as  a ${\rm Bin}(n-1, p)$ and Binomials concentrate around theirs means,   on average we may expect that $\bE_{n,p}[\tau]\approx\frac{n (n-1)p}{(n-1)p}=n$.   Hence, in the dense case, we expect the function $\bE_{n,p}[\tau]$ to be a constant function of $p$ for large values of $n$, so that the occupation measure should be close to the prior measure $\mu$. This is the content of our first main result, Theorem \ref{thm:p-constant}.

\begin{theorem}\label{thm:p-constant}
Let $\mu$ be a fixed probability distribution on $(0,1]$. For any Borel measurable set $A\subseteq (0,1]$, define $\mu_{n,\infty}(A)$ as in \eqref{def_occupation_measure_G_N_p}. Then as $n\to\infty$,
$$
\mu_{n,\infty}(A)\to \mu_{\infty,\infty}(A)=\mu(A)\,.
$$

\end{theorem}


\subsection{Sparse case}
In the sparse case, the edge probability $p\in (0,1]$ scales with the number of vertices $n$ of the Erdös-Rényi random graph as follows: $p=\lambda/n$ where $\lambda$ belongs to a compact set $\Lambda\subset \mathbb{R}_{>0}$. We suppose also that $n$ is large enough so that $\lambda/n\in (0,1]$ for all $\lambda\in\Lambda.$
Setting $\Theta=\Lambda$, $\cF=\mathcal{B}(\Lambda)$ where $\mathcal{B}(\Lambda)$ denotes the Borel sigma-algebra of $\Lambda$ and $g_n=Id/n$ with $Id$ denoting the identity function, the occupation measure of $\lambda$ is given by
\begin{equation}
\label{def_occupation_measure_G_n_lambda_over_n}
\mu_{n,\infty}(A)= \frac{\int_{A}\bE_{n,\lambda/n}[\tau]\mu(d\lambda)}{\int_{\Lambda}\bE_{n,\lambda/n}[\tau]\mu(d\lambda)}\,.
\end{equation} 

In the sparse case, contrarily to the dense case, the behavior of the Erdös-Rényi random graph $G(n,\lambda/n)$ changes drastically as a function of $\lambda$. 
As we show in our second main result, in this case, the occupation measure differs from the prior measure $\mu$ and concentrates on values of $\lambda>1$ (supercritical values for the underlying Erdös-Rényi random graph).
Before stating Theorem \ref{thm:p-lambda_over_n}, let us introduce a few more notation.

In what follows, for each $\lambda>0$, we denote by $\zeta_{\lambda}$ the {\it survival probability} of a branching process with Poisson offspring  distribution of parameter $\lambda>0$ and by $\eta_\lambda=1-\zeta_\lambda$ the {\it extinction probability}. Recall that $\zeta_{\lambda}=0$ for all $0<\lambda\leq 1$, whereas 
$\zeta_{\lambda}>0$ for each $\lambda>1$ (see for instance  \cite[Chapter~3]{hofstad_2016}.) Finally, let us define for $\lambda> 0$,
\begin{equation}\label{def:2nd_moment_reciprocal_of_Poisson}
R_{\lambda}:=\mathbb{E}\left[\frac{1}{(1+{\rm Poi}(\lambda))^2}\right]=\frac{e^{-\lambda}}{\lambda} \int_{0}^{\lambda}(e^{s}-1)s^{-1}ds\,,
\end{equation}
where the last equality is known, see, e.g., \cite{10.2307/2284399}. With this notation, Theorem \ref{thm:p-lambda_over_n} reads as follows.

\begin{theorem}\label{thm:p-lambda_over_n}
Let $\Lambda$ be a compact subset of $\bR_{>0}$ and $\mu$ be a probability distribution on $\Lambda$ such that $\mu(\Lambda \cap (1,\infty)) > 0$. For any Borel measurable set $A\subseteq \Lambda$, let $\mu_{n,\infty}(A)$ be as in \eqref{def_occupation_measure_G_n_lambda_over_n}. Then as $n\to\infty$,
\begin{equation*}
\label{def_occupation_measure_at_n=infty}
\mu_{n,\infty}(A)\to \mu_{\infty,\infty}(A):= \frac{\int_{A}f(\lambda)\mu(d\lambda)}{\int_{\Lambda}f(\lambda)\mu(d\lambda)}\,,
\end{equation*}
where the function $f:[0,\infty)\to [0,\infty)$ is given by
\begin{equation*}
 f(\lambda)=\lambda^2\zeta_{\lambda}\left(R_{\lambda}+\eta_{\lambda}(1+\eta_{\lambda})(R_{\lambda} - \eta^2_{\lambda}R_{\lambda\eta_{\lambda}})\right)\,.   
\end{equation*}
In particular, 
$f(\lambda)=0$ for all $\lambda\in\Lambda\cap (0,1]$ and
$\text{supp}\left(\mu_{\infty,\infty}\right)=\Lambda\cap (1,\infty)$.
\end{theorem}

\begin{remark}
It is worth mentioning that $f(\lambda)=\lim_{n\to\infty}n^{-1}\bE_{n,\lambda/n}[\tau]$ and that $f(\lambda)>0$ for any $\lambda\in \Lambda\cap (1,\infty)$. The former will be shown in the proof of Theorem \ref{thm:p-lambda_over_n} and the latter follows from \eqref{def:2nd_moment_reciprocal_of_Poisson} and the fact that 
 the extinction probability is a solution to $\eta_{\lambda}=e^{-\lambda(1-\eta_{\lambda})}$(see \cite[Chapter~3, Equation 3.6.31]{hofstad_2016}) which then imply that
$$
\lambda (R_{\lambda}-\eta^2_{\lambda}R_{\lambda\eta_{\lambda}})=e^{-\lambda}\int_{\lambda\eta_{\lambda}}^{\lambda}(e^{s}-1)s^{-1}ds\geq 0\,, 
$$
so that $f(\lambda)\geq \zeta_{\lambda}\lambda^2R_{\lambda}>0$.
\end{remark}

\begin{remark}
 The proofs of both Theorems 1 and 2 heavily rely on the homogeneity of the expected degree of the vertices and do not generalize directly to random graph models with heterogeneous degree distribution as, for example, in the Chung-Lu model \cite{chung2002average}.   
\end{remark}

\section{Proof of the main results} \label{sec:proof}

Recall that, by \eqref{def_undicional_exp_of_return_time_to_1}, $\bE_{n,p}\left[\tau\right]=\bE_{n,p}\left[\bE_{G(n,p)}\left[\tau\right]\right]$. In the sequel, let $d(1)$ denote the degree of vertex $1$ and  $|E(\cC(1))|$ denote the number of edges in the random graph $G(n,p)$ induced by the vertices in $\cC(1)$. 
By our convention on $\tau$, $\bE_{G(n,p)}\left[\tau\right]=1$ on the event $\{|\cC(1)|=1\}=\{d(1)=0\}$.
Moreover, since the random walk always starts at vertex $1$, on the complementary event $\{|\cC(1)|>1\}=\{d(1)\geq 1\}$, we obtain that  
\begin{equation*}
\bE_{G(n,p)}\left[\tau\right]=\frac{2|E(\cC(1))|}{d(1)}\,.
\end{equation*}
 The above equation follows from a classical result relating the stationary distribution of Markov chains to the expected return times (see, e.g., \cite[Proposition~1.19]{levin2017markov}).
By introducing the random variable,
\begin{equation}
\label{def:ind_deg_larger_than_1_over_deg}
\frac{\mathbbm{1}_{\{d(1)\geq 1  \}}}{d(1)}:=
\begin{cases}
0\,,  & \text{ if } \ d(1)=0\\
\frac{1}{d(1)}\,, & \text{ if } \ d(1)\geq 1
\end{cases},
\end{equation}
we can rewrite $\bE_{G(n,p)}\left[\tau\right]$ as
{\begin{equation}\label{eq:stationary}
\bE_{G(n,p)}\left[\tau\right]= 2|E(\cC(1))|\frac{\mathbbm{1}_{\{d(1)\geq 1  \}}}{d(1)} +\mathbbm{1}_{\{d(1)=0\}} \,,
\end{equation}
where $\mathbbm{1}_A$ denotes the indicator function of the event $A$.  
Equation~\eqref{eq:stationary} will be used to prove Lemma \ref{lem:main} below which will be important to prove our main results, namely Theorem~\ref{thm:p-constant} and Theorem~\ref{thm:p-lambda_over_n}. 
\begin{lemma}\label{lem:main}
For all $n\geq 4$ and for all $p\in (0,1)$, it holds that 
\begin{align*}
\frac{1}{n-1}\left(\bE_{n,p}\left[\tau\right] -1\right) = & \mathbb{E}_{n,p}\bigg[d(2)\frac{\mathbbm{1}_{\{d(1)\geq 1  \}}}{d(1)}\bigg] 
\\
&- \mathbb{P}_{n,p}\left( 2 \notin \cC(1) \right) \mathbb{E}_{n,p}\bigg[d(2)\frac{\mathbbm{1}_{\{d(1)\geq 1  \}}}{d(1)}\Big| 2 \notin \cC(1)\bigg]\,.
\end{align*}
Moreover, 
\begin{align*}
&\mathbb{E}_{n,p}\bigg[d(2)\frac{\mathbbm{1}_{\{d(1)\geq 1  \}}}{d(1)}\bigg]=  \mathbb{E}\left[\frac{(n-2)^2p^2 (1-p)}{(1+B_{n-3})^2}\right] 
+   \frac{1 + (n-2)p}{(n-1)}\left( 1-(1-p)^{n-1}\right)
\\&
\mathbb{E}_{n,p}\bigg[d(2)\frac{\mathbbm{1}_{\{d(1)\geq 1  \}}}{d(1)}\Big| 2 \notin \cC(1) \bigg]=  \mathbb{E}_{n,p}\bigg[\frac{\mathbbm{1}_{\{d(1)\geq 1  \}}}{d(1)} \frac{p(n-1-|\cC(1)|) (n-|\cC(1)|)}{n-1}  \bigg]
\,,
\end{align*}
where $B_{n-3}$ denotes a random variable with distribution  ${\rm Bin}(n-3,p)$.
\end{lemma}

\begin{proof}
Fix $n\geq 4$ and $p\in (0,1)$. Since  $2|E(\cC(1))|= \sum_{x \in \cC(1)}d(x)$, we may write 
\begin{align*} 
2|E(\cC(1))|\frac{\mathbbm{1}_{\{d(1)\geq 1  \}}}{d(1)}
&=\mathbbm{1}_{\{d(1)\geq 1  \}} +  \sum_{x =2}^{n}\mathbbm{1}_{\{x \in \cC(1) \}}d(x)\frac{\mathbbm{1}_{\{d(1)\geq 1  \}}}{d(1)}\,, 
\end{align*}
which, together with \eqref{def_undicional_exp_of_return_time_to_1} and \eqref{eq:stationary}, implies 
\begin{equation}\label{eq:decomposition}
\begin{split}
&\mathbb{E}_{n,p}\left[\tau \right]  = 1  +\sum_{x=2}^{n}\mathbb{E}_{n,p}\bigg[\mathbbm{1}_{\{x \in \cC(1) \}}d(x)\frac{\mathbbm{1}_{\{d(1)\geq 1  \}}}{d(1)}\bigg]
\\
&
= 1 +(n-1)\mathbb{E}_{n,p}\bigg[\mathbbm{1}_{\{2 \in \cC(1) \}}d(2)\frac{\mathbbm{1}_{\{d(1)\geq 1  \}}}{d(1)}\bigg]\\
& =1 +(n-1)\Bigg(\underbrace{\mathbb{E}_{n,p}\bigg[d(2)\frac{\mathbbm{1}_{\{d(1)\geq 1  \}}}{d(1)}\bigg]}_{a)} - \underbrace{\mathbb{E}_{n,p}\bigg[\mathbbm{1}_{\{2 \notin \cC(1) \}}d(2)\frac{\mathbbm{1}_{\{d(1)\geq 1  \}}}{d(1)}\bigg]}_{b)}\Bigg)
\,,
\end{split}
\end{equation}
where the second equality is due to the symmetry of the model. 
For the  term $a)$ in \eqref{eq:decomposition}, we have that 
\begin{multline*}
\mathbb{E}_{n,p}\bigg[d(2)\frac{\mathbbm{1}_{\{d(1)\geq 1  \}}}{d(1)}\bigg] =    
 \mathbb{E}_{n,p}\bigg[d(2)\frac{\mathbbm{1}_{\{d(1)\geq 1  \}}}{d(1)}\mathbbm{1}_{\{ \omega_{\{1,2\}}=0   \}}\bigg]   +
 \mathbb{E}_{n,p}\bigg[d(2)\frac{\mathbbm{1}_{\{d(1)\geq 1  \}}}{d(1)}\mathbbm{1}_{\{ \omega_{\{1,2\}}=1   \}}\bigg]  
\,, 
\end{multline*}
where $\{\omega_{\{1,2\}}=1\}$ (resp.,  $\{\omega_{\{1,2\}}=0\}$) denotes the event that the edge between $1$ and $2$ is present (resp., absent). 
Conditionally on the event $\{ \omega_{\{1,2\}}=0  \}$, we have that  $d(2)$ and $d(1)$ are independent and identically distributed and the distribution is ${\rm Bin}(n-2, p)$. Since $\frac{\mathbbm{1}_{\{d(1)\geq 1  \}}}{d(1)}$ is a function of $d(1)$ only, we also have that $d(2)$ and $\frac{\mathbbm{1}_{\{d(1)\geq 1  \}}}{d(1)}$ are conditionally independent given the event $\{\omega_{\{1,2\}}=0\}$. Now, conditionally on the event $\{ \omega_{\{1,2\}}=1   \}$, we have that  $d(2)=1+B_{n-2}$ and $d(1)=1+B_{n-2}'$ with $B_{n-2}$ and $B_{n-2}'$  independent and identically distributed with distribution ${\rm Bin}(n-2, p)$. 
In this case, it follows that $\frac{\mathbbm{1}_{\{d(1)\geq 1  \}}}{d(1)}=1/(B'_{n-2}+1)$ which is also independent of $d(2)$.  
Therefore, we have that 
\begin{align*}
 \mathbb{E}_{n,p}&\bigg[d(2)\frac{\mathbbm{1}_{\{d(1)\geq 1  \}}}{d(1)} \mathbbm{1}_{\{ \omega_{\{1,2\}}=0   \}}\bigg] =  \mathbb{E}_{n,p}\bigg[d(2)\frac{\mathbbm{1}_{\{d(1)\geq 1  \}}}{d(1)}  \bigg\vert  \omega_{\{1,2\}}=0   \bigg] (1-p) 
 \\
 & = \mathbb{E}\bigg[B_{n-2}\bigg]\mathbb{E}\bigg[\frac{\mathbbm{1}_{\{d(1)\geq 1  \}}}{d(1)}\bigg\vert  \omega_{\{1,2\}}=0\bigg](1-p)
 \\
 &
 = (n-2)p(1-p)\left((n-2)p\mathbb{E}\left[\frac{1}{(1+B_{n-3})^2}\right]\right)
 =(n-2)^2p^2(1-p)\mathbb{E}\left[\frac{1}{(1+B_{n-3})^2}\right]\,,
\end{align*}
where $B_{n-3}\sim{\rm Bin}(n-3,p)$ and in the third equality we used Lemma~\ref{lem:bin-moments} (see Appendix~\ref{sec:appendix}).
Similarly, 
\begin{align*}
    &\mathbb{E}_{n,p}\bigg[d(2)\frac{\mathbbm{1}_{\{d(1)\geq 1  \}}}{d(1)} \mathbbm{1}_{\{ \omega_{\{1,2 \}}=1   \}}\bigg] =  \mathbb{E}_{n,p}\bigg[d(2)\frac{\mathbbm{1}_{\{d(1)\geq 1  \}}}{d(1)}\bigg\vert  \omega_{\{1,2\}}=1  \bigg]p
    \\
    &
    =\mathbb{E}\bigg[\frac{1+B_{n-2}}{1+B_{n-2}'}\bigg]p = \mathbb{E}\bigg[1+B_{n-2}\bigg]\mathbb{E}\bigg[\frac{1}{1+B_{n-2}'}\bigg]p= \big(1 + (n-2)p\big) \frac{ 1-(1-p)^{n-1}}{(n-1)}\,, 
\end{align*}
where in the last equality we used that $\mathbb{E}\left[\frac{1}{1+B_{n-2}'}\right] = \frac{1-(1-p)^{n-1}}{p(n-1)}$ (see, e.g., \cite[Equation~3.4]{10.2307/2284399}).

As far as  term $b)$ in \eqref{eq:decomposition} is concerned, we have that 
\begin{align*}
&\mathbb{E}_{n,p}\bigg[\mathbbm{1}_{\{2 \notin \cC(1) \}}d(2)\frac{\mathbbm{1}_{\{d(1)\geq 1  \}}}{d(1)}\bigg]=
\sum_{\ell =1}^{n-1}\mathbb{E}_{n,p}\bigg[\mathbbm{1}_{\{|\cC(1)|=\ell \}}\mathbbm{1}_{\{2 \notin \cC(1) \}}d(2)\frac{\mathbbm{1}_{\{d(1)\geq 1  \}}}{d(1)}\bigg]
\\
&=\sum_{\ell =1}^{n-1}\mathbb{P}_{n,p}\left(|\cC(1)|=\ell, 2 \notin \cC(1) \right)\mathbb{E}_{n,p}\bigg[d(2)\frac{\mathbbm{1}_{\{d(1)\geq 1  \}}}{d(1)}\bigg\vert |\cC(1)|=\ell, 2 \notin \cC(1)  \bigg]
\\
&
\overset{i)}{=}
\sum_{\ell =1}^{n-1}\mathbb{P}_{n,p}\left(|\cC(1)|=\ell, 2 \notin \cC(1) \right)\mathbb{E}_{n,p}\bigg[\frac{\mathbbm{1}_{\{d(1)\geq 1  \}}}{d(1)} \bigg\vert |\cC(1)|=\ell, 2 \notin \cC(1)  \bigg] p\left(n-1 - \ell \right) 
\\
&
=
p\mathbb{E}_{n,p}\bigg[\frac{\mathbbm{1}_{\{d(1)\geq 1  \}}\mathbbm{1}_{\{2 \notin \cC(1)\}}}{d(1)} (n-1-|\cC(1)|)   \bigg] 
\overset{ii)}{=} p \mathbb{E}_{n,p}\bigg[\frac{\mathbbm{1}_{\{d(1)\geq 1  \}}(n-|\cC(1)|)}{d(1) (n-1)} (n-1-|\cC(1)|)   \bigg] 
\,,
\end{align*}
where, in $i)$ we used that conditioned on  $\{2 \notin \cC(1)\}$ and $\{|\cC(1)|=\ell\}$, $d(2)$ and $d(1)$ are independent and that 
$\mathbb{E}_{n,p}\bigg[d(2) \Big\vert |\cC(1)|=\ell, 2 \notin \cC(1)  \bigg]=p(n-1-\ell)$. 
In $ii)$ we used the symmetry of the model 
and the fact that $\sum_{v=2}^n \mathbbm{1}_{\{v \notin \cC(1)\}}=n-|\cC(1)|$.

\end{proof}

\begin{lemma}\label{Lem:exp_tau_uniformly_bounded}
There exists a constant $C>0$ such that for all $n\geq 2$ and $p\in [0,1]$, it holds that 
$$
\frac{1}{n-1}\bE_{n,p}\left[\tau\right]\leq C\,.
$$
\end{lemma}
\begin{proof}
The cases $p=0$ and $p=1$ are straightforward. Now take $p\in (0,1)$.
By Lemma \ref{lem:main}, it suffices to show that for all $m\geq 2$ and $p\in (0,1)$,
\begin{equation}\label{ineq_1_proof_lem_exp_tau_uniformly_bounded}
\bE\left[\frac{(mp)^2}{(1+B_m)^2}\right]\leq K\,,    
\end{equation}
for some universal constant $K>0$, where $B_m$ denotes a random variable with distribution $\text{Bin}(m,p)$. To show \eqref{ineq_1_proof_lem_exp_tau_uniformly_bounded}, we first use  the multiplicative Chernoff bound (see, e.g., \cite[Theorem~4.5]{10.5555/3134214}) to deduce that
\begin{equation}\label{ineq_2_proof_lem_exp_tau_uniformly_bounded}
\bP(B_m\leq mp/2)\leq e^{-mp/8}\,.
\end{equation}
Then, by denoting $B=\{B_m\leq mp/2\}$, it follows that $(mp)^2/(1+B_m)^2\leq 4$ on $B^c$ so that 
$$
\bE\left[\frac{(mp)^2}{(1+B_m)^2}\right]\leq 4\bP(B^c)+(mp)^2e^{-mp/8}\,,
$$
where we have also used inequality \eqref{ineq_2_proof_lem_exp_tau_uniformly_bounded} and the trivial bound $1/(1+x)^2\leq 1$ for any $x\geq 0$. Finally, using that the function $[0,\infty) \ni y\mapsto g(y)=y^2e^{-y/8}$ attains its maximum at $y=16$ and that $\bP(B^c)\leq 1$, we obtain that 
$$
\bE\left[\frac{(mp)^2}{(1+B_m)^2}\right]\leq 4+g(16)=K\,,
$$
proving Inequality \eqref{ineq_1_proof_lem_exp_tau_uniformly_bounded}.
\end{proof}

\subsection{Proof of Theorem~\ref{thm:p-constant} (dense case)} 
The proof of Theorem~\ref{thm:p-constant} follows from Lemmas~\ref{lem:main} and \ref{Lem:exp_tau_uniformly_bounded}, and the following observations:

\begin{itemize}
    \item[1.] For every fixed $p\in (0,1)$ and every $n\geq 2$, setting $q=1-p$, it holds that  
    \[
(2-q^{n-2})q^{n-1}\leq \mathbb{P}_{n,p}\left(2\notin \cC(1)\right) \leq  2q^{n-1} \left(1+q^{\frac{n-2}{2}} \right)^{n-2}\,. 
\]
In particular, 
$\mathbb{P}_{n,p}\left(2\notin \cC(1)\right) \approx 2 (1-p)^{n-1}$ asymptotically in $n$ (see, \cite[Theorem~2]{10.1214/aoms/1177706098}).
\item[2.] If $\{X_n\}_{n\geq 1}$ is a sequence of random variables with $X_n\sim {\rm Bin}(n,p)$, then for every $\alpha>0$ 
\[
\lim_{n \to +\infty} \frac{\mathbb{E}\left[(1+X_n)^{-\alpha}\right]}{\left(1+\mathbb{E}[X_n]\right)^{-\alpha}}=1\,.
\]
Specifically, $\mathbb{E}\left[(1+X_n)^{-2}\right] \approx \frac{1}{(1+np)^2}$ asymptotically in $n$ (see, \cite[Theorem~2.1]{GARCIA2001235}).
\end{itemize}

\begin{proof}[Proof of Theorem~\ref{thm:p-constant}]

By Lemma \ref{Lem:exp_tau_uniformly_bounded} and the Bounded Convergence Theorem, in order to prove the claim, it suffices to show that, for every fixed $p \in (0,1]$, 
\begin{align*}
\lim_{n \to + \infty} \frac{1}{n-1}\bE_{n,p}\left[\tau\right]= 1\,.    \end{align*}
The case $p=1$ is straightforward since under $\bP_{n,1}$, we know that $|\cC(1)|=n$ and $d(1)=n-1=d(2)$ almost surely, so that  $\mathbb{P}_{n,1}\left( 2 \notin \cC(1) \right)=0$  and 
$\bE_{n,1}[(d(2)/d(1))\mathbbm{1}_{d(1)\geq 1}]=\mathbb{P}_{n,1}(d(1)\geq 1)=1$. By Lemma ~\ref{lem:main} it then follows that $(n-1)^{-1}\bE_{n,1}[\tau]=n/(n-1)$ implying  that $\lim_{n \to + \infty} (n-1)^{-1}\bE_{n,1}[\tau]=1$. 

Now take $p\in (0,1).$ In this case, we will show that
\begin{align*}
&\lim_{n \to + \infty}\mathbb{E}_{n,p}\bigg[d(2)\frac{\mathbbm{1}_{\{d(1)\geq 1  \}}}{d(1)}\bigg]=1\,, \tag{A}
\\
&\lim_{n \to + \infty}\mathbb{P}_{n,p}\left( 2 \notin \cC(1) \right) \mathbb{E}_{n,p}\bigg[d(2)\frac{\mathbbm{1}_{\{d(1)\geq 1  \}}}{d(1)}\Big| 2 \notin \cC(1)\bigg] =0\,. \tag{B}
\end{align*}

By item $2.$ above, we have that \[
\lim_{n \to + \infty} (n-2)^2p^2 (1-p) \mathbb{E}\left[\frac{1}{(1+B_{n-3})^2}\right] = \lim_{n \to + \infty}  \frac{(n-2)^2p^2 (1-p)}{(1+(n-3)p)^2}= 1-p\,.
\]
Moreover, since $
\lim_{n \to + \infty}  \frac{1 + (n-2)p}{(n-1)}\left( 1-(1-p)^{n-1}\right) = p$, from Lemma~\ref{lem:main} we obtain (A).
As regards to (B), by Lemma~\ref{lem:main} we have that 
\begin{align*}
\limsup_{n \to + \infty}  \mathbb{P}_{n,p}\left(2 \notin \cC(1) \right) &  \mathbb{E}_{n,p}\bigg[\frac{\mathbbm{1}_{\{d(1)\geq 1  \}}}{d(1)} \frac{p(n-1-|\cC(1)|) (n-|\cC(1)|)}{n-1}  \bigg]
\\
&
\leq \limsup_{n \to + \infty}  \mathbb{P}_{n,p}\left(2 \notin \cC(1) \right)  (n-1)p\,. 
\end{align*}
%
By item $1.$ 
we then obtain that
$\limsup_{n \to + \infty}  \mathbb{P}_{n,p}\left(2 \notin \cC(1) \right)(n-1)p=   0
 $,
proving (B).
\end{proof}

\subsection{Proof of Theorem~\ref{thm:p-lambda_over_n} (sparse case)}
\medskip 
\noindent 

{\bf Notation:} 
For a sequence $(a_n)_{n\geq 1}$ of real numbers and a sequence $(b_n)_{n\geq 1}$ of positive real numbers, we write $a_n=O(b_n)$ to indicate that $|a_n|\leq C b_n$ for all $n\geq 1$ and for some positive constant $C>0$.

\medskip 
We begin with a few auxiliary results which will be used in the proof of Theorem~\ref{thm:p-lambda_over_n}.
%
\begin{theorem}[Theorem~4.8 and Theorem~5.4 of \cite{hofstad_2016}]\label{prop:LLN_Cmax}{\ \\}

{\rm A)} Fix $\lambda>1$. Then, for every $\nu \in (1/2,1)$, there exists $\delta=\delta(\lambda, \nu)>0$ such that 
\[
\mathbb{P}_{n, \lambda/n}\left(\big| |\cC_{\rm max}|-\zeta_\lambda n\big|\geq n^{\nu} \right)=O(n^{-\delta})\,.
\]

{\rm B)} Set $\lambda=1$. Then,  for any $n>1000$ and $A>8$, it holds that
\[
\mathbb{P}_{n, 1/n}\left( |\cC(1)|\geq An^{2/3} \right)\leq 4 e^{-\frac{A^2(A-4)}{32}} n^{-1/3}\,.
\]
\end{theorem}

\begin{proposition}\label{prop:P-conditioned}
Fix $\lambda >1$. For each $m\geq 1$, 
\begin{align*}
\lim_{n\to+\infty }\bP_{n,\lambda/n}\left(1\notin\cC_{\rm max}\bigg\vert d(1)=m\right)=
(1-\zeta_{\lambda})^m\,.
\end{align*}  
\end{proposition}
The result of Proposition~\ref{prop:P-conditioned} should be known but we could not find a reference. For completeness, a proof is provided in Appendix \ref{App:proof_prop_p_conditioned}.
%
In the next result, we show that in the subcritical regime the size of the connected component of a given vertex has all the moments uniformly bounded.
\begin{lemma}
\label{lem_upper_bound_size_cluster_subcritical_regime}
Fix $\lambda\in (0,1)$. Then, for any $k\geq 1$ it holds that 
\begin{equation*}
\sup_{n\geq 1} \left\{  \bE_{n,\lambda/n}\left[|\cC(1)|^k \right]\right\}\leq \sum_{m=0}^{\infty}(m+1)^ke^{-mI_{\lambda}}<+\infty\,,   
\end{equation*}
where $I_{\lambda}=\lambda-1-\log(\lambda)>0$. Moreover, $\bP_{n,\lambda/n}(2\in\cC(1))=O(1/n)$.
\end{lemma}

\begin{proof}
The tail sum formula for moments implies that for each $n\geq 2$ and $k\geq 1$,
 \begin{align*}
 &\bE_{n,\lambda/n}\left[|\cC(1)|^k\right] =\sum_{m=0}^{n-1}\left((m+1)^k-m^k\right)\bP_{n,\lambda/n}\left(|\cC(1)|>m\right)\\
 & \leq \sum_{m=0}^{n-1}\left((m+1)^k-m^k\right) e^{-mI_{\lambda}}\leq \sum_{m=0}^{\infty}\left((m+1)^k-m^k\right) e^{-mI_{\lambda}}\leq \sum_{m=0}^{\infty}(m+1)^k e^{-mI_{\lambda}}\,,
 \end{align*}
where the first inequality follows from the fact that each connected component is stochastically dominated by the total progeny of a branching process with binomial offspring distribution and a large deviation principle for Binomial random variables (see, e.g.,  \cite[Section~4.3.2 and Inequality~4.3.11]{hofstad_2016}).  
%
%
Now, by using that 
$$
\bP_{n,\lambda/n}(2\in\cC(1))=\left(\bE_{n,\lambda/n}\left[|\cC(1)|\right]-1\right)(n-1)^{-1}\,,
$$
and the first part of the proof (for $k=1$),  we obtain  $\bP_{n,\lambda/n}(2\in\cC(1))=O(1/n)$.
 \end{proof}

In the lemma below, we show that in the supercritical regime the $k$-th moment of the fraction of vertices belonging to the connected component of a given vertex converges to $\zeta^{k+1}_{\lambda}$ as the number of vertices increases.

\begin{lemma}\label{lem:expected-C(1)}
Fix $\lambda>1$. Then, for any $k\geq 1$ it holds that 
\begin{align*}
\lim_{n\to+\infty}\frac{1}{n^k}\bE_{n,\lambda/n}\bigg[ |\cC(1)|^k \bigg]=\zeta_{\lambda}^{k+1}\,.    
\end{align*}
Moreover, for any  $m\geq 1$ and $k\geq 1$, it holds that
$$
\lim_{n\to+\infty }\frac{1}{n^k}\bE_{n,\lambda/n}\left[|\cC(1)|^k\bigg\vert d(1)=m\right]=\zeta^k_{\lambda}(1-(1-\zeta_{\lambda})^m)\,.
$$ 
\end{lemma}
\begin{proof}
We begin with the proof of the second statement. Observe that
\begin{align*}
\frac{1}{n^k}\mathbb{E}_{n,\lambda/n}\bigg[|\cC(1)|^k\mathbbm{1}_{\{d(1)=m\}}\bigg] & = 
\frac{1}{n^k}\mathbb{E}_{n,\lambda/n}\bigg[|\cC_{\rm max}|^k\mathbbm{1}_{\{d(1)=m\}}\mathbbm{1}_{\{1\in\cC_{\rm max}\}}\bigg]
\\
&+
\frac{1}{n^k}\mathbb{E}_{n,\lambda/n}\bigg[|\cC(1)|^k\mathbbm{1}_{\{d(1)=m\}}\mathbbm{1}_{\{1\notin\cC_{\rm max}\}}\bigg] \,. 
\end{align*} 
The second term is $o(1)$ since, on the event  $\{1 \notin \cC_{\rm max}\}$ the size of $\cC(1)$ is at most the size of the second largest component. From Theorem~5.4 in \cite{janson2011random}, the second largest component is at most $K\log n$ with high probability (with $K$ an explicit constant depending on $\lambda$), specifically, if $\cC_{(2)}$ denotes the second largest component, it holds that  $\mathbb{P}_{n,\lambda/n}(|\cC_{(2)}|> K\log n)\to 0$ as $n$ tends to infinity.

It remains to study the first term. Given $\nu \in (1/2,1)$, let us define the event $A_n=\{| |\cC_{\rm max}| - \zeta_\lambda n|\leq n^\nu\}$. We can then write 
\begin{align*}
\frac{1}{n^k}\mathbb{E}_{n,\lambda/n}\bigg[|\cC_{\rm max}|^k\mathbbm{1}_{\{d(1)=m\}}\mathbbm{1}_{\{1\in\cC_{\rm max}\}}\bigg] &= 
\frac{1}{n^k}\mathbb{E}_{n,\lambda/n}\bigg[|\cC_{\rm max}|^k\mathbbm{1}_{\{d(1)=m\}}\mathbbm{1}_{\{1\in\cC_{\rm max}\}} \mathbbm{1}_{A_n}\bigg]
\\
&+\frac{1}{n^k}\mathbb{E}_{n,\lambda/n}\bigg[|\cC_{\rm max}|^k\mathbbm{1}_{\{d(1)=m\}}\mathbbm{1}_{\{1\in\cC_{\rm max}\}} \mathbbm{1}_{A_n^c}\bigg]\,.
\end{align*}
Since $|\cC_{\rm max}|/n\leq 1$, by Theorem~\ref{prop:LLN_Cmax}-{\rm A)} we have that 
\[
\mathbb{E}_{n,\lambda/n}\bigg[\left(\frac{|\cC_{\rm max}|}{n}\right)^k\mathbbm{1}_{\{d(1)=m\}}\mathbbm{1}_{\{1\in\cC_{\rm max}\}} \mathbbm{1}_{A_n^c}\bigg]
\leq \mathbb{P}_{n,\lambda/n}(A_n^c)= O(n^{-\delta})\,.
\]
Using  again Theorem~\ref{prop:LLN_Cmax}-{\rm A)} and Proposition~\ref{prop:P-conditioned} together with the following upper and lower bounds
\begin{align*}
\frac{1}{n^k}\mathbb{E}_{n,\lambda/n}\bigg[|\cC_{\rm max}|^k\mathbbm{1}_{\{d(1)=m\}}\mathbbm{1}_{\{1\in\cC_{\rm max}\}} \mathbbm{1}_{A_n}\bigg]
\begin{cases}
\leq (\frac{\zeta_\lambda n + n^\nu}{n})^k\mathbb{E}_{n,\lambda/n}\bigg[\mathbbm{1}_{\{d(1)=m  \}} \mathbbm{1 }_{\{1 \in \cC_{\rm max} \}} \mathbbm{1}_{A_n}\bigg],
\\
\\
\geq (\frac{\zeta_\lambda n - n^\nu}{n})^k\mathbb{E}_{n,\lambda/n}\bigg[\mathbbm{1}_{\{d(1)=m  \}} \mathbbm{1 }_{\{1 \in \cC_{\rm max} \}} \mathbbm{1}_{A_n}\bigg],
\end{cases}
\end{align*}
the claim  follows.
The proof of the first limit goes along the same lines of reasoning. First we write 
\begin{align*}
\frac{1}{n^k}\mathbb{E}_{n,\lambda/n}\bigg[ |\cC(1)|^k\bigg] & = 
\frac{1}{n^k}\mathbb{E}_{n,\lambda/n}\bigg[|\cC_{\rm max}|^k \mathbbm{1 }_{\{1 \in \cC_{\rm max} \}}\bigg] + \frac{1}{n^k}\mathbb{E}_{n,\lambda/n}\bigg[|\cC(1)|^k\mathbbm{1 }_{\{1 \notin \cC_{\rm max} \}}\bigg]\,.    
\end{align*} 
The last term on the right-hand side   converges to zero (by Theorem~5.4 in \cite{janson2011random} as before). By  Theorem~\ref{prop:LLN_Cmax}-{\rm A)} and the fact that $\lim_{n\to +\infty}\mathbb{P}_{n,\lambda/n}(1 \in \cC_{\rm max})=\zeta_\lambda$, we obtain that the first term on the right-hand side converges to $\zeta_\lambda^{k+1}$ as claimed.
\end{proof}

\begin{lemma}\label{lem:_moments_of_C1}
Fix $\lambda>1$. Then, for any $k\geq 1$ it holds that
\begin{align*}
\lim_{n\to+ \infty}
\frac{1}{n^k}\mathbb{E}_{n,\lambda/n}\bigg[\frac{\mathbbm{1}_{\{d(1)\geq 1  \}}}{d(1)} |\cC(1)|^k\bigg] 
= \zeta_{\lambda}^k\lambda\left(R_{\lambda} - \eta_{\lambda}^2R_{\lambda\eta_{\lambda}}\right)\,,    
\end{align*}  
where  $R_{\lambda}$ is defined in \eqref{def:2nd_moment_reciprocal_of_Poisson}. 
\end{lemma}
\begin{remark}
The Poisson limit theorem and Lemma~\ref{lem:bin-moments} (see Appendix~\ref{sec:appendix}) imply that for all $\lambda>0$,
$$
\lim_{n\to+\infty}\bE_{n,\lambda/n}\bigg[ \frac{\mathbbm{1}_{\{d(1)\geq 1  \}}}{d(1)}\bigg]=\lambda \bE\left[\frac{1}{(1+{\rm Poi}(\lambda))^2}\right]=\lambda R_{\lambda}\,.
$$
Hence, from Lemma~\ref{lem:_moments_of_C1}, it follows that the random variables $\mathbbm{1}_{\{d(1)\geq 1  \}}/d(1)$ and $(|\cC(1)|/n)^k$ are, whenever $\lambda>1$, asymptotically  
correlated. 
As a matter of fact, 
\begin{align*}
\lim_{n\to \infty}&
\frac{1}{n^k}\mathbb{E}_{n,\lambda/n}\bigg[\frac{\mathbbm{1}_{\{d(1)\geq 1  \}}}{d(1)} |\cC(1)|^k\bigg]-\bE_{n,\lambda/n}\bigg[ \frac{\mathbbm{1}_{\{d(1)\geq 1  \}}}{d(1)}\bigg]\frac{1}{n^k}\bE_{n,\lambda/n}\bigg[ |\cC(1)|^k \bigg]
\\
&= \zeta_{\lambda}^k\lambda\left(R_{\lambda} - R_{\lambda\eta_{\lambda}}\eta_{\lambda}^2\right) - \zeta_\lambda^{k+1}\lambda R_{\lambda}=\zeta_{\lambda}^k\eta_\lambda\left(\lambda R_{\lambda} - \lambda\eta_{\lambda} R_{\lambda\eta_{\lambda}}\right).
\end{align*}
\end{remark}

\begin{proof} 
For each $m\geq 1$, denote $p_{\lambda}(m)=e^{-\lambda}\lambda^m/m!$ and observe that 
\begin{equation}
\label{proof_moments_c1_eq_1}
\begin{split}
\frac{1}{n^k}\mathbb{E}_{n,\lambda/n}\bigg[\frac{\mathbbm{1}_{\{d(1)\geq 1  \}}}{d(1)} |\cC(1)|^k\bigg]= \sum_{m=1}^{n-1}\frac{1}{m}p_{\lambda}(m)\bE_{n,\lambda/n}\left[\left(\frac{|\cC(1)|}{n}\right)^k\bigg\vert d(1)=m\right]\\
+\sum_{m=1}^{n-1}\frac{1}{m}(\bP_{n,\lambda/n}\left(d(1)=m\right)-p_{\lambda}(m))\bE_{n,\lambda/n}\left[\left(\frac{|\cC(1)|}{n}\right)^k\bigg\vert d(1)=m\right]\,.
\end{split}
\end{equation}
Using the estimate below (see, e.g.,   \cite[Theorem~1]{barbour_hall_1984}) 
$$
\sum_{m=1}^{n-1}\left|p_{\lambda}(m)-\bP_{n,\lambda/n}\left(d(1)=m\right)\right|\leq 2\lambda(1-e^{-\lambda})(n-1)^{-1}\,,
$$
it follows from \eqref{proof_moments_c1_eq_1} and from the fact that $|\cC(1)|/n\leq 1$, that
\begin{equation}
\label{proof_moments_c1_eq_2}
\frac{1}{n^k}\mathbb{E}_{n,\lambda/n}\bigg[\frac{\mathbbm{1}_{\{d(1)\geq 1  \}}}{d(1)} |\cC(1)|^k\bigg]=\sum_{m=1}^{n-1}\frac{p_{\lambda}(m)}{m}\bE_{n,\lambda/n}\left[\left(\frac{|\cC(1)|}{n}\right)^k\bigg\vert d(1)=m\right]+O(1/n)\,.
\end{equation}
%
%
Using Lemma~\ref{lem:expected-C(1)} and the Dominated convergence theorem, from~\eqref{proof_moments_c1_eq_2} we obtain that
\begin{equation}
\begin{split}
\label{proof_moments_c1_eq_3}
\lim_{n\to+\infty }\frac{1}{n^k}\mathbb{E}_{n,\lambda/n}\bigg[\frac{\mathbbm{1}_{\{d(1)\geq 1  \}}}{d(1)} |\cC(1)|^k\bigg]&=\zeta^k_{\lambda}\sum_{m=1}^{\infty}\frac{1}{m}p_{\lambda}(m)(1-(1-\zeta_{\lambda})^m)\\
&=\zeta^k_{\lambda}\left(\lambda R_{\lambda}-\sum_{m=1}^{\infty}\frac{1}{m}p_{\lambda}(m)(1-\zeta_{\lambda})^m\right)\,,
\end{split}
\end{equation}
where in the second equality we used that $\sum_{m=1}^{\infty}m^{-1}p_{\lambda}(m)=\lambda\sum_{m=0}^{\infty}(1+m)^{-2}p_{\lambda}(m)=\lambda R_{\lambda}$.
Using that $e^{-\lambda(1-\eta_{\lambda})}=\eta_{\lambda}$ and the fact that $$\sum_{m=1}^{\infty}m^{-1}p_{\lambda}(m)\eta_{\lambda}^m=\lambda \eta_{\lambda}e^{-\lambda(1-\eta_{\lambda})}\sum_{m=0}^{\infty}(1+m)^{-2}p_{\lambda\eta_{\lambda}}(m)=\lambda\eta^2_{\lambda}R_{\lambda\eta_{\lambda}}\,,$$
the claim follows. 
\end{proof}

\begin{proof}[Proof of Theorem~\ref{thm:p-lambda_over_n}]
Lemma~\ref{lem:main} with $p=\lambda/n$ gives
 \begin{multline}
 \label{identity_1_proof_theorem_p_lambda_over_n}
 \left(\bE_{n,\lambda/n}[\tau]-1\right)(n-1)^{-1}= \bE_{n,\lambda/n}\left[d(2)\frac{\mathbbm{1}_{\{d(1)\geq 1  \}}}{d(1)}\right]-\\
\bP_{n,\lambda/n}\left(2\notin \cC(1)\right)\mathbb{E}_{n,\lambda/n}\bigg[d(2)\frac{\mathbbm{1}_{\{d(1)\geq 1  \}}}{d(1)}\Big| 2\notin \cC(1)\bigg]\,,    
 \end{multline}
 where,
 \begin{multline}
 \label{identity_2_proof_theorem_p_lambda_over_n}
 \bE_{n,\lambda/n}\left[d(2)\frac{\mathbbm{1}_{\{d(1)\geq 1  \}}}{d(1)}\right]=\left(\frac{\lambda}{n}\right)^2\left(1-\frac{\lambda}{n}\right)(n-2)^2\bE\left[\frac{1}{(1+B_{n-3})^2}\right]\\
 +\left(1+(n-2)\frac{\lambda}{n}\right)\left(1-\left(1-\frac{\lambda}{n}\right)^{n-1}\right)(n-1)^{-1}\,,    
 \end{multline}
and 
\begin{multline}
\label{identity_3_proof_theorem_p_lambda_over_n}
 \bE\bigg[d(2)\frac{\mathbbm{1}_{\{d(1)\geq 1  \}}}{d(1)}\Big| 2\notin \cC(1)\bigg]=
 \lambda \mathbb{E}\bigg[\frac{\mathbbm{1}_{\{B_{n-1}\geq 1\}}}{B_{n-1}}\bigg]\\
 +\frac{\lambda}{n}\mathbb{E}_{n,\lambda/n}\bigg[\frac{\mathbbm{1}_{\{d(1)\geq 1  \}}}{d(1)} \frac{|\cC(1)|^2 }{n-1}  \bigg]-\frac{\lambda(2n-1)}{n}\bE_{n,\lambda/n}\left[\frac{\mathbbm{1}_{\{d(1)\geq 1\}}}{d(1)} \frac{|\cC(1)|}{(n-1)}\right]\,,
\end{multline}
where $B_{n-1}\sim{\rm Bin}(n-1,\lambda/n)$ and $B_{n-3}\sim{\rm Bin}(n-3,\lambda/n)$.
From \eqref{identity_2_proof_theorem_p_lambda_over_n} we have that, for all $\lambda>0$
$$
\lim_{n\to+\infty}\bE_{n,\lambda/n}\bigg[ d(2)\frac{\mathbbm{1}_{\{d(1)\geq 1  \}}}{d(1)}\bigg]=\lambda^2  \bE\left[\frac{1}{(1+{\rm Poi}(\lambda))^2}\right]=\lambda^2 R_{\lambda}\,.
$$
For the first term on the right-hand side of \eqref{identity_3_proof_theorem_p_lambda_over_n}, from Lemma~\ref{lem:bin-moments} (see Appendix~\ref{sec:appendix}) and the Poisson limit theorem we have that, for all $\lambda>0$,  
\[
\lim_{n\to +\infty}  \lambda \mathbb{E}\bigg[\frac{\mathbbm{1}_{\{B_{n-1}\geq 1\}}}{B_{n-1}}\bigg] = \lambda^2 R_\lambda\,.
\]
For the other two terms on the right-hand side of \eqref{identity_3_proof_theorem_p_lambda_over_n} it holds that 
\begin{align*}
\lim_{n\to +\infty} \frac{\lambda}{n}\mathbb{E}_{n,\lambda/n}\bigg[\frac{\mathbbm{1}_{\{d(1)\geq 1  \}}}{d(1)} \frac{|\cC(1)|^2 }{n-1}  \bigg] =\begin{cases}
\lambda^2\zeta_{\lambda}^2\left(R_{\lambda} - R_{\lambda\eta_{\lambda}}\eta_{\lambda}^2\right)\,, & \text{ if }\lambda >1\,,
\\
0\,, &\text{ if }\lambda \leq 1\,,
\end{cases}
\end{align*}
and 
\begin{align*}
\lim_{n\to +\infty} \frac{\lambda(2n-1)}{n}\mathbb{E}_{n,\lambda/n}\bigg[\frac{\mathbbm{1}_{\{d(1)\geq 1  \}}}{d(1)} \frac{|\cC(1)|}{n-1}  \bigg] =\begin{cases}
2\lambda^2\zeta_{\lambda}\left(R_{\lambda} - R_{\lambda\eta_{\lambda}}\eta_{\lambda}^2\right)\,, & \text{ if }\lambda >1\,,
\\
0\,, &\text{ if }\lambda \leq 1\,,
\end{cases}
\end{align*}
where, the result in the supercritical regime follows from Lemma~\ref{lem:_moments_of_C1}, in the subcritical from Lemma~\ref{lem_upper_bound_size_cluster_subcritical_regime} and in the critical regime from noticing that, if we define $E=\{|C(1)|<9n^{2/3}\}$, we can write
\[
\frac{1}{n}\mathbb{E}_{n,1/n}\bigg[\frac{\mathbbm{1}_{\{d(1)\geq 1  \}}}{d(1)} \frac{|\cC(1)|^2}{n-1} \mathbbm{1}_E \bigg] + \frac{1}{n}\mathbb{E}_{n,1/n}\bigg[\frac{\mathbbm{1}_{\{d(1)\geq 1  \}}}{d(1)} \frac{|\cC(1)|^2}{n-1} \mathbbm{1}_{E^c} \bigg]\,,
\]
and the first term converges to zero since $|C(1)|$ is at most $9n^{2/3}$, as well as the second since the whole term is bounded from above by  $\mathbb{P}_{n,1/n}(E^c)$ which by Theorem~\ref{prop:LLN_Cmax}-{\rm B)}, goes to zero as $n$ tends to infinity (the second limit in the critical regime is analogous).  
Using that $\zeta_\lambda=0$ for $\lambda\leq 1$  we can simply write that, for all $\lambda>0$, 
\begin{align*}
&\lim_{n\to +\infty} \frac{\lambda}{n}\mathbb{E}_{n,\lambda/n}\bigg[\frac{\mathbbm{1}_{\{d(1)\geq 1  \}}}{d(1)} \frac{|\cC(1)|^2 }{n-1}  \bigg] =\lambda^2
\zeta_{\lambda}^2\left(R_{\lambda} -R_{\lambda\eta_{\lambda}}\eta_{\lambda}^2\right)
\\
&
\lim_{n\to +\infty} \frac{\lambda(2n-1)}{n}\mathbb{E}_{n,\lambda/n}\bigg[\frac{\mathbbm{1}_{\{d(1)\geq 1  \}}}{d(1)} \frac{|\cC(1)|}{n-1}  \bigg] =
2\lambda^2\zeta_{\lambda}\left(R_{\lambda} - R_{\lambda\eta_{\lambda}}\eta_{\lambda}^2\right)\,.
\end{align*}
Moreover, by using that $\bP_{n,\lambda/n}(2\in\cC(1))=\left(\bE_{n,\lambda/n}\left[|\cC(1)|\right]-1\right)(n-1)^{-1}$, together with Lemma~\ref{lem:_moments_of_C1} (supercritical regime), Lemma~\ref{lem_upper_bound_size_cluster_subcritical_regime} (subcritical regime), Theorem~\ref{prop:LLN_Cmax}-{\rm B)} (critical regime) and the fact that $\zeta_\lambda=0$ for $\lambda\leq 1$, we have that, for all $\lambda>0$ 
\[
\lim_{n \to +\infty}\bP_{n,\lambda/n}(2\notin\cC(1))= 1 -\zeta_\lambda=\eta_{\lambda}\,. 
\]
Using the above results in \eqref{identity_1_proof_theorem_p_lambda_over_n}, we finally obtain that, for all $\lambda>0$ 
\begin{align*}
\lim_{n \to +\infty}&\left(\bE_{n,\lambda/n}[\tau]-1\right)(n-1)^{-1}
\\
&=\lambda^2 R_\lambda
- \eta_\lambda\left( \lambda^2 R_\lambda + \lambda^2\zeta_{\lambda}^2\left(R_{\lambda} - R_{\lambda\eta_{\lambda}}\eta_{\lambda}^2\right) - 
2\lambda^2\zeta_{\lambda}\left(R_{\lambda} - R_{\lambda\eta_{\lambda}}\eta_{\lambda}^2\right) \right)
\\
&=\lambda^2\zeta_{\lambda}\left(R_{\lambda}(1+2\eta_{\lambda}-\zeta_{\lambda}\eta_{\lambda}) -\eta^2_{\lambda}R_{\lambda\eta_{\lambda}}(2\eta_{\lambda}-\zeta_{\lambda}\eta_{\lambda})\right)\\
&=\lambda^2\zeta_{\lambda}\left(R_{\lambda}(1+\eta_{\lambda}+\eta^2_{\lambda}) -\eta^2_{\lambda}R_{\lambda\eta_{\lambda}}(\eta_{\lambda}+\eta^2_{\lambda})\right)
=f(\lambda)\,,
\end{align*}
where in the second and third equalities we used that $\eta_{\lambda}=(1-\zeta_{\lambda})$. 
Hence, the result follows from Lemma \ref{Lem:exp_tau_uniformly_bounded} and the Bounded Convergence Theorem.
\end{proof}

\appendix
\section{Auxiliary results} \label{sec:appendix}

\begin{lemma}\label{lem:bin-moments}
 Let $m>1$,  $B_m\sim {\rm Bin}(m,p)$ and define
 \begin{equation*}
\frac{\mathbbm{1}_{\{B_m\geq 1  \}}}{B_m}:=
\begin{cases}
0 \,,  & \text{ if } \ B_m=0\\
\frac{1}{B_m}\,, & \text{ if } \ B_m\geq 1
\end{cases}.
\end{equation*}
 If $B_{m-1}\sim{\rm Bin}(m-1,p)$, then it holds that 
 \[
 \mathbb{E}\left[\frac{\mathbbm{1}_{\{B_m\geq 1  \}}}{B_m}\right]= mp \mathbb{E}\left[\frac{1}{(1+B_{m-1})^2}\right]\,.
 \]
\end{lemma}
\begin{proof}
By the definition of the random variable $\frac{\mathbbm{1}_{\{B_m\geq 1  \}}}{B_m}$, we have that
\begin{align*}
\mathbb{E}&\left[\frac{\mathbbm{1}_{\{B_m\geq 1  \}}}{B_m}\right] =\sum_{k=1}^m \frac{1}{k}\binom{m}{k} p^k (1-p)^{m-k}=mp \sum_{k=1}^m \frac{1}{k^2}\binom{m-1}{k-1} p^{k-1} (1-p)^{m-k}
 \\
 &= mp \sum_{i=0}^{m-1} \frac{1}{(1+i)^2}\binom{m-1}{i} p^{i} (1-p)^{m-1-i}= mp \mathbb{E}\left[\frac{1}{(1+B_{m-1})^2}\right]\,.
 \end{align*}
\end{proof}

\begin{lemma}\label{lem:exp-moments}
 Let $G$ be a finite connected graph and $(X_n)_{n\geq 0}$ be a simple symmetric random walk on $G$. For $x$ a vertex of $G$, let $\tau_x$ denote the first return time to $x$, {\em i.e., } $\tau_x:=\inf\{n\geq 1: X_n = x\}$. Then, there exists a positive constant $a=a(G)>0$ such that 
 \[
 \mathbb{E}_x [e^{a\tau_x}]<+\infty\,. 
 \]
\end{lemma}
The finiteness of the exponential moments of $\tau_x$ implies that $\mathbb{E}_x[\tau_x^m]<\infty$, $\forall m\geq 1$.
The proof of Lemma~\ref{lem:exp-moments} is a minor variation of the proof of Lemma~1.13 in \cite{levin2017markov} and we provide it here for completeness. 
\begin{proof}
Since $G$ is connected and finite, by irreducibility of $(X_n)_{n\geq 0}$, there exist $r=r(G)>0$ and $\varepsilon=\varepsilon(G)>0$ such that for every $\ell\geq 1$ it holds that 
\[
\mathbb{P}_x(\tau_x>\ell r)\leq (1-\varepsilon)^\ell\,.
\]
Using the above,  for any $\beta>0$, we obtain that
\begin{align*}
\mathbb{E}_x[e^{\beta \tau_x}]&= \sum_{n=1}^\infty \mathbb{P}_x(\tau_x=n)e^{\beta n}=\sum_{k=0}^\infty \sum_{n=k r + 1}^{(k+1)r} \mathbb{P}_x(\tau_x=n)e^{\beta n}
\\
&
\leq \sum_{k=0}^\infty e^{\beta (k+1)r}\sum_{n=k r + 1}^{(k+1)r} \mathbb{P}_x(\tau_x=n) \leq \sum_{k=0}^\infty e^{\beta (k+1)r}\mathbb{P}_x(\tau_x>kr)   
\\
&
\leq e^{\beta r}\sum_{k=0}^\infty e^{\beta kr} (1-\varepsilon)^k\,.  
\end{align*}
Choosing $\beta<-\frac{\log (1-\varepsilon)}{r}$, we obtain that $\mathbb{E}_x[e^{\beta \tau_x}]<\infty$. 
\end{proof}

%
\section{Proof of Proposition
\ref{prop:LLN_for_the_empirical_occupation_measure} and 
\ref{prop:CLT_for_the_empirical_occupation_measure}
}
\label{Sec:proof_of_prop_LLN}

Here, we prove Proposition~\ref{prop:LLN_for_the_empirical_occupation_measure} and 
\ref{prop:CLT_for_the_empirical_occupation_measure}. Their proofs follow classical arguments of the theory of Compound Renewal processes (e.g., see Theorem 1.2.3 and Theorem 1.3.3 of \cite{Borovkov2022Compound}).

\begin{proof}[Proof of Proposition~\ref{prop:LLN_for_the_empirical_occupation_measure}]
For each $T\geq 1$, define $M_T:=\inf\{k\geq 1: \tau_1+\ldots+\tau_k\geq T\}$. Since the random variables $(\tau_i)_{i\geq 1}$ are i.i.d. with finite positive mean,  the integer-valued random variable $M_T$ is such that  $M_T\to \infty$ almost surely as $T\to\infty$. 
Using the definition of $M_T$ and the convention that a sum over an empty set is 0, the empirical occupation measure $\mu_{n,T}(A)$ can be rewritten as
\begin{equation*}
 \mu_{n,T}(A)=\frac{1}{T}\sum_{t=1}^{M_T-1}\tau_t\mathbbm{1}_{\{\theta_t\in A\}}+\left(\frac{T-\sum_{t=1}^{M_T-1}\tau_t}{T}\right)\mathbbm{1}_{\{\theta_{M_T}\in A\}}.    
\end{equation*}
Next, use that $\sum_{t=1}^{M_T-1}\tau_t<T\leq  \sum_{t=1}^{M_T}\tau_t$ to deduce that
\begin{equation}
\label{ineq_1_proof_LLN}
\frac{\sum_{t=1}^{M_T-1}\tau_t\mathbbm{1}_{\{\theta_t\in A\}}}{\sum_{t=1}^{M_T}\tau_t}\leq \mu_{n,T}(A)\leq \frac{\sum_{t=1}^{M_T-1}\tau_t\mathbbm{1}_{\{\theta_t\in A\}}}{\sum_{t=1}^{M_T-1}\tau_t}+\left(\frac{T-\sum_{t=1}^{M_T-1}\tau_t}{T}\right),
\end{equation}
where on the RHS we also used  the trivial bound $\mathbbm{1}_{\{\theta_{M_T}\in A\}}\leq 1$.
Now, Lemma \ref{lem:exp-moments} ensures that the random variables $\tau_i$ have finite moment of any order. Hence, since $(T-\sum_{t=1}^{M_T-1}\tau_t)\leq \tau_{M_T}$, Theorem 8.2 of Chapter 6 of \cite{Gut2013Probability} implies that the second term of the RHS \eqref{ineq_1_proof_LLN} converges almost surely to $0$ as $T\to\infty$. To conclude the proof, observe that the LHS and the first term of the RHS of \eqref{ineq_1_proof_LLN} converge almost surely to $\mu_{n,T}(A)$ by the Strong Law of Large numbers. 
\end{proof}



\begin{proof}[Proof of Proposition~\ref{prop:CLT_for_the_empirical_occupation_measure}]
For each $T\geq 1$, let $M_T$ be as in the proof of Proposition \ref{prop:LLN_for_the_empirical_occupation_measure} and observe that one can decompose 
$\sqrt{T}(\mu_{n,T}(A)-\mu_{n,\infty}(A))$ as
\begin{equation*}
\frac{1}{\sqrt{T}}\sum_{t=1}^{M_T-1}\tau_t\left(\mathbbm{1}_{\{\theta_t\in A\}}-\mu_{n,\infty}(A)\right)+\frac{(T-\sum_{t=1}^{M_T-1}\tau_t)}{\sqrt{T}}\left(\mathbbm{1}_{\{\theta_{M_T}\in A\}}-\mu_{n,\infty}(A)\right).    
\end{equation*}
Arguing as in Proposition \ref{prop:LLN_for_the_empirical_occupation_measure}, one can show that the second term above converges  almost surely to $0$ as $T\to\infty$.
We then conclude the proof by using Anscombe's Theorem (e.g., see Theorem 3.2 of Chapter 7 of \cite{Gut2013Probability}). 
\end{proof}


\section{Proof of Proposition \ref{prop:P-conditioned}}
\label{App:proof_prop_p_conditioned}
\begin{proof}[Proof of Proposition~\ref{prop:P-conditioned}]
We proceed by induction in $m$. For  $m=1$, notice that
$$
\bP_{n,\lambda/n}\left(1\in\cC_{\rm max}\mid d(1)=1\right)=\frac{\bP_{n,\lambda/n}\left(1\in\cC_{\rm max}, d(1)=1\right)}{{n-1\choose 1}(\lambda/n)(1-\lambda/n)^{n-2}}\,.
$$
For each $i, j\in \{1,\ldots, n\}$ with $i\neq j$, let us denote $\{w_{\{i,j\}}=1\}$ (resp., $\{w_{\{i,j\}}=0\}$) the event that the edge between vertices $i$ and $j$ is open (resp., closed). With this notation, for any $i\geq 2$, write $A_{i}= \{w_{\{1,i\}}=1\}\cap\bigcap_{j=2:j\neq i}^{n}\{w_{\{1,j\}}=0\}$,  and observe that
\begin{align*}
\bP_{n,\lambda/n}&\left(1\in\cC_{\rm max}, d(1)=1\right)=\sum_{i=2}^{n}\bP_{n,\lambda/n}\left(A_i\right)\bP_{n,\lambda/n}\left(1\in\cC_{\rm max}\mid A_i\right)\\
&=(\lambda/n)(1-\lambda/n)^{n-2}\sum_{i=2}^{n}\bP_{n,\lambda/n}\left(1\in\cC_{\rm max}\mid A_i\right)\,,
\end{align*}
which implies  that
$$
\bP_{n,\lambda/n}\left(1\in\cC_{\rm max}\mid d(1)=1\right)=\frac{1}{n-1}\sum_{i=2}^{n}\bP_{n,\lambda/n}\left(1\in\cC_{\rm max}\mid A_i\right)\,.
$$
Since for any $i\geq 2$, $\bP_{n,\lambda/n}\left(1\in\cC_{\rm max}\mid A_i\right)=\bP_{n,\lambda/n}\left(i\in\cC_{\rm max}\mid A_i\right)=\bP_{n-1,\lambda/n}\left(1\in\cC_{\rm max}\right)$ 
we conclude that
$$
\bP_{n,\lambda/n}\left(1\in\cC_{\rm max}\mid d(1)=1\right)=\bP_{n-1,\lambda/n}\left(1\in\cC_{\rm max}\right)\,.
$$
By the monotonicity of the model and the fact that $\lim_{n\to +\infty}\mathbb{P}_{n,\lambda/n}(1 \in \cC_{\rm max})=\zeta_\lambda$, it follows that 
$$
\limsup_{n\to\infty }\bP_{n-1,\lambda/n}\left(1\in\cC_{\rm max}\right)\leq \limsup_{n\to\infty }\bP_{n,\lambda/n}\left(1\in\cC_{\rm max}\right)=\zeta_{\lambda}\,.
$$
On the other hand,  for any fixed $\epsilon>0$ satisfying  $\lambda(1-\epsilon)>1$ there exists $n_{0}=n_{0}(\epsilon)$ such that for all $n\geq n_0$, it holds $\lambda(n-1)/n\geq \lambda(1-\epsilon)>1$. Using the monotonocity of the model once more, we have that
$$
\liminf_{n\to\infty }\bP_{n-1,\lambda/n}\left(1\in\cC_{\rm max}\right)\geq \liminf_{n\to\infty }\bP_{n-1,\lambda(1-\epsilon)/(n-1)}\left(1\in\cC_{\rm max}\right)=\zeta_{\lambda(1-\epsilon)}\,.
$$
By taking the limit $\epsilon\to 0$ and using the continuity of $\zeta_{\lambda}$ (see, e.g.,~\cite[Corollary 3.19]{hofstad_2016}), we then deduce that 
$$
\lim_{n\to\infty}\bP_{n,\lambda/n}\left(1\in\cC_{\rm max}\mid d(1)=1\right)=\lim_{n\to\infty}\bP_{n-1,\lambda/n}\left(1\in\cC_{\rm max}\right)=\zeta_{\lambda}\,,
$$
establishing the claim for $m=1$. Now, suppose the claim  holds for $m\geq 1$. We will show that it also holds form $m+1$.
By using similar arguments as before, one can show that
\begin{equation}
\label{proof_moments_c1_eq_4}
\bP_{n,\lambda/n}\left(1\notin\cC_{\rm max}\mid d(1)=m+1\right)=\bP_{n-1,\lambda/n}\left(1\notin\cC_{\rm max},\ldots,m+1\notin\cC_{\rm max}\right)\,.
\end{equation}
By writing
\begin{multline*}
\bP_{n-1,\lambda/n}\left(1\notin\cC_{\rm max},\ldots,m+1\notin\cC_{\rm max}\right)=\bP_{n-1,\lambda/n}\left(1\notin\cC_{\rm max},\ldots,m\notin\cC_{\rm max}\right)\\-\bP_{n-1,\lambda/n}\left(1\notin\cC_{\rm max},\ldots,m\notin\cC_{\rm max},m+1\in\cC_{\rm max}\right)\,,
\end{multline*}
and then using the symmetry of the model to write,
\begin{multline*}
\bP_{n-1,\lambda/n}\left(1\notin\cC_{\rm max},\ldots,m\notin\cC_{\rm max},m+1\in\cC_{\rm max}\right)\\
=\frac{1}{n-m-1}\bE_{n-1,\lambda/n}\left[\prod_{i=1}^m\mathbbm{1}_{\{i\notin\cC_{\rm max}\}}|\cC_{\rm max}|\right]\,.
\end{multline*}
Consider the event $A_n=\{| |\cC_{\rm max}| - \zeta_\lambda n|\leq n^\nu\}$ for some $\nu \in (1/2,1)$. By Theorem~\ref{prop:LLN_Cmax}-{\rm A)} we have that 
\begin{multline*}
\frac{1}{n-m-1}\bE_{n-1,\lambda/n}\left[\prod_{i=1}^m\mathbbm{1}_{\{i\notin\cC_{\rm max}\}}|\cC_{\rm max}|\right]\\=\frac{1}{n-m-1}\bE_{n-1,\lambda/n}\left[\prod_{i=1}^m\mathbbm{1}_{\{i\notin\cC_{\rm max}\}}|\cC_{\rm max}|\mathbbm{1}_{A_{n-1}}\right]+O((n-1)^{-\delta})\,.
\end{multline*}
Using the inequality
\begin{multline*}
\bE_{n-1,\lambda/n}\left[\frac{\prod_{i=1}^m\mathbbm{1}_{\{i\notin\cC_{\rm max}\}}|\cC_{\rm max}|}{(n-m-1)}\mathbbm{1}_{A_{n-1}}\right]
\\
\leq \frac{\zeta_\lambda (n-1) + (n-1)^\nu}{n-m-1}\mathbb{E}_{n-1,\lambda/n}\bigg[\prod_{i=1}^m\mathbbm{1 }_{\{i \notin \cC_{\rm max} \}}\bigg]\,,
\end{multline*}
together with the inductive hypothesis, we then deduce that
\begin{multline*} 
\liminf_{n\to\infty }\bP_{n-1,\lambda/n}\left(1\notin\cC_{\rm max},\ldots,m+1\notin\cC_{\rm max}\right)\geq (1-\zeta_{\lambda})^m-\zeta_{\lambda}(1-\zeta_{\lambda})^{m}\\
=(1-\zeta_{\lambda})^{m+1}\,.
\end{multline*}
By proceeding first for fixed $\epsilon>0$ such that $\lambda(1-\epsilon)>1$ as in the case $m=1$ and then by taking $\epsilon\to 0$,  one can show that
\begin{multline*}
\liminf_{n\to\infty}\frac{1}{n-m-1}\bE_{n-1,\lambda/n}\left[\prod_{i=1}^m\mathbbm{1}_{\{i\notin\cC_{\rm max}\}}|\cC_{\rm max}|\right]\\
\geq \zeta_{\lambda}\lim_{n\to \infty}\bP_{n-1,\lambda/n}\left(1\notin\cC_{\rm max},\ldots,m\notin\cC_{\rm max}\right)=\zeta_{\lambda}(1-\zeta_{\lambda})^m,
\end{multline*} 
so that by the inductive hypothesis, 
\begin{multline*}
\limsup_{n\to\infty}\bP_{n-1,\lambda/n}\left(1\notin\cC_{\rm max},\ldots,m+1\notin\cC_{\rm max}\right)\leq (1-\zeta_{\lambda})^m\\
-\liminf_{n\to\infty}\frac{1}{n-m-1}\bE_{n-1,\lambda/n}\left[\prod_{i=1}^m\mathbbm{1}_{\{i\notin\cC_{\rm max}\}}|\cC_{\rm max}|\right]\leq (1-\zeta_{\lambda})^m-\zeta_{\lambda}(1-\zeta)^m\\
=(1-\zeta_{\lambda})^{m+1}\,.
\end{multline*}
As a consequence, it follows (recall identity \eqref{proof_moments_c1_eq_4}) that
\begin{multline*}
\lim_{n\to\infty}\bP_{n,\lambda/n}\left(1\notin\cC_{\rm max}\mid d(1)=m+1\right)=\\
\lim_{n\to\infty}\bP_{n-1,\lambda/n}\left(1\notin\cC_{\rm max},\ldots,m+1\notin\cC_{\rm max}\right)
=(1-\zeta_{\lambda})^{m+1}\,,
\end{multline*}
 which proves the claim.

\end{proof}



 \bibliographystyle{unsrt}
 \bibliography{ref}

{\bf Acknowledgments} 
G.I. and G.O. thank the Brain Institute - UFRN for the hospitality  during a visit   that initiated this work.
G.O. thanks FAPERJ (grant E-26/201.397/2021), CNPq (grant 303166/2022-3) and Serrapilheira Institute
(grant number Serra – 2211-42049) for financial support. G.I.~was supported by FAPERJ (grant E-26/210.516/2024). D.Y.T. thanks the Institute of Mathematics - UFRJ for the hospitality and UFRN for  financial support (grants 002/2021-11.87.00.00.00.1912 and 008/2022-11.87.00.00.00.44).



\end{document}